\documentclass[12pt]{amsart}       
\usepackage{txfonts}
\usepackage{amssymb}
\usepackage{eucal}
\usepackage{graphicx}
\usepackage{amsmath}
\usepackage{amscd}
\usepackage[all]{xy}           
\usepackage{tikz}
\usepackage{amsfonts,latexsym}
\usepackage{xspace}
\usepackage{epsfig}
\usepackage{float}
\usepackage{color}
\usepackage{fancybox}
\usepackage{colordvi}
\usepackage{multicol}
\usepackage{colordvi}
\usepackage[colorlinks,final,backref=page,hyperindex,hypertex]{hyperref}
\usepackage[active]{srcltx} 

\newtheorem{theorem}{Theorem}[section]
\theoremstyle{definition}
\newtheorem{defn}[theorem]{Definition}
\newtheorem{lemma}[theorem]{Lemma}
\newtheorem{coro}[theorem]{Corollary}
\newtheorem{prop-def}{Proposition-Definition}[section]
\newtheorem{coro-def}{Corollary-Definition}[section]

\newtheorem{remark}[theorem]{Remark}

\newcommand{\nc}{\newcommand}
\nc{\tred}[1]{\textcolor{red}{#1}}
\nc{\tblue}[1]{\textcolor{blue}{#1}}
\nc{\tgreen}[1]{\textcolor{green}{#1}}
\nc{\tpurple}[1]{\textcolor{purple}{#1}}
\nc{\btred}[1]{\textcolor{red}{\bf #1}}
\nc{\btblue}[1]{\textcolor{blue}{\bf #1}}
\nc{\btgreen}[1]{\textcolor{green}{\bf #1}}
\nc{\btpurple}[1]{\textcolor{purple}{\bf #1}}
\nc{\NN}{{\mathbb N}}
\nc{\ncsha}{{\mbox{\cyr X}^{\mathrm NC}_\lambda}} \nc{\ncshao}{{\mbox{\cyr
X}^{\mathrm NC}_\lambda}}

\newcommand{\efootnote}[1]{}

\renewcommand{\textbf}[1]{}

\newcommand{\delete}[1]{}

\nc{\mlabel}[1]{\label{#1}}  
\nc{\mcite}[1]{\cite{#1}}  
\nc{\mref}[1]{\ref{#1}}  
\nc{\mbibitem}[1]{\bibitem{#1}} 

\delete{
\nc{\mlabel}[1]{\label{#1}  
{\hfill \hspace{1cm}{\small\tt{{\ }\hfill(#1)}}}}
\nc{\mcite}[1]{\cite{#1}{\small{\tt{{\ }(#1)}}}}  
\nc{\mref}[1]{\ref{#1}{{\tt{{\ }(#1)}}}}  
\nc{\mbibitem}[1]{\bibitem[\bf #1]{#1}} 
}

\nc{\opa}{\ast} \nc{\opb}{\odot} \nc{\op}{\bullet} \nc{\pa}{\frakL}
\nc{\arr}{\rightarrow} \nc{\lu}[1]{(#1)} \nc{\mult}{\mrm{mult}}
\nc{\diff}{\mathfrak{Diff}}
\nc{\opc}{\sharp}\nc{\opd}{\natural}
\nc{\ope}{\circ}
\nc{\dpt}{\mathrm{d}}
\nc{\diam}{alternating\xspace}
\nc{\Diam}{Alternating\xspace}
\nc{\cdiam}{alternating\xspace}
\nc{\Cdiam}{Alternating\xspace}
\nc{\AW}{\mathcal{A}}
\nc{\rba}{Rota-Baxter algebra\xspace}

\nc{\ari}{\mathrm{ar}}

\nc{\lef}{\mathrm{lef}}

\nc{\Sh}{\mathrm{ST}}

\nc{\Cr}{\mathrm{Cr}}

\nc{\st}{{Schr\"oder tree}\xspace}
\nc{\sts}{{Schr\"oder trees}\xspace}

\nc{\vertset}{\Omega} 

\nc{\assop}{\quad \begin{picture}(5,5)(0,0)
\line(-1,1){10}
\put(-2.2,-2.2){$\bullet$}
\line(0,-1){10}\line(1,1){10}
\end{picture} \quad \smallskip}

\nc{\operator}{\begin{picture}(5,5)(0,0)
\line(0,-1){6}
\put(-2.6,-1.8){$\bullet$}
\line(0,1){9}
\end{picture}}

\nc{\idx}{\begin{picture}(6,6)(-3,-3)
\put(0,0){\line(0,1){6}}
\put(0,0){\line(0,-1){6}}
 \end{picture}}

\nc{\pb}{{\mathrm{pb}}}
\nc{\Lf}{{\mathrm{Lf}}}

\nc{\lft}{{left tree}\xspace}
\nc{\lfts}{{left trees}\xspace}

\nc{\fat}{{fundamental averaging tree}\xspace}

\nc{\fats}{{fundamental averaging trees}\xspace}
\nc{\avt}{\mathrm{Avt}}

\nc{\rass}{{\mathit{RAss}}}

\nc{\aass}{{\mathit{AAss}}}

\nc{\vin}{{\mathrm Vin}}    
\nc{\lin}{{\mathrm Lin}}    
\nc{\inv}{\mathrm{I}n}
\nc{\gensp}{V} 
\nc{\genbas}{\mathcal{V}} 
\nc{\bvp}{V_P}     
\nc{\gop}{{\,\omega\,}}     

\nc{\bin}[2]{ (_{\stackrel{\scs{#1}}{\scs{#2}}})}  
\nc{\binc}[2]{ \left (\!\! \begin{array}{c} \scs{#1}\\
    \scs{#2} \end{array}\!\! \right )}  
\nc{\bincc}[2]{  \left ( {\scs{#1} \atop
    \vspace{-1cm}\scs{#2}} \right )}  
\nc{\bs}{\bar{S}} \nc{\cosum}{\sqsubset} \nc{\la}{\longrightarrow}
\nc{\rar}{\rightarrow} \nc{\dar}{\downarrow} \nc{\dprod}{**}
\nc{\dap}[1]{\downarrow \rlap{$\scriptstyle{#1}$}}
\nc{\md}{\mathrm{dth}} \nc{\uap}[1]{\uparrow
\rlap{$\scriptstyle{#1}$}} \nc{\defeq}{\stackrel{\rm def}{=}}
\nc{\disp}[1]{\displaystyle{#1}} \nc{\dotcup}{\
\displaystyle{\bigcup^\bullet}\ } \nc{\gzeta}{\bar{\zeta}}
\nc{\hcm}{\ \hat{,}\ } \nc{\hts}{\hat{\otimes}}
\nc{\barot}{{\otimes}} \nc{\free}[1]{\bar{#1}}
\nc{\uni}[1]{\tilde{#1}} \nc{\hcirc}{\hat{\circ}} \nc{\lleft}{[}
\nc{\lright}{]} \nc{\lc}{\lfloor} \nc{\rc}{\rfloor}
\nc{\curlyl}{\left \{ \begin{array}{c} {} \\ {} \end{array}
    \right .  \!\!\!\!\!\!\!}
\nc{\curlyr}{ \!\!\!\!\!\!\!
    \left . \begin{array}{c} {} \\ {} \end{array}
    \right \} }
\nc{\longmid}{\left | \begin{array}{c} {} \\ {} \end{array}
    \right . \!\!\!\!\!\!\!}
\nc{\onetree}{\bullet} \nc{\ora}[1]{\stackrel{#1}{\rar}}
\nc{\ola}[1]{\stackrel{#1}{\la}}
\nc{\ot}{\otimes} \nc{\mot}{{{\boxtimes\,}}}
\nc{\otm}{\overline{\boxtimes}} \nc{\sprod}{\bullet}
\nc{\scs}[1]{\scriptstyle{#1}} \nc{\mrm}[1]{{\rm #1}}
\nc{\margin}[1]{\marginpar{\rm #1}}   
\nc{\dirlim}{\displaystyle{\lim_{\longrightarrow}}\,}
\nc{\invlim}{\displaystyle{\lim_{\longleftarrow}}\,}
\nc{\mvp}{\vspace{0.3cm}} \nc{\tk}{^{(k)}} \nc{\tp}{^\prime}
\nc{\ttp}{^{\prime\prime}} \nc{\svp}{\vspace{2cm}}
\nc{\vp}{\vspace{8cm}} \nc{\proofbegin}{\noindent{\bf Proof: }}
\nc{\proofend}{$\blacksquare$ \vspace{0.3cm}}
\nc{\modg}[1]{\!<\!\!{#1}\!\!>}
\nc{\intg}[1]{F_C(#1)} \nc{\lmodg}{\!
<\!\!} \nc{\rmodg}{\!\!>\!}
\nc{\cpi}{\widehat{\Pi}}
\nc{\sha}{{\mbox{\cyr X}}}  
\nc{\shap}{{\mbox{\cyrs X}}} 
\nc{\shpr}{\diamond}    
\nc{\shp}{\ast} \nc{\shplus}{\shpr^+}
\nc{\shprc}{\shpr_c}    
\nc{\msh}{\ast} \nc{\zprod}{m_0} \nc{\oprod}{m_1}
\nc{\vep}{\varepsilon} \nc{\labs}{\mid\!} \nc{\rabs}{\!\mid}
\nc{\sqmon}[1]{\langle #1\rangle}

\nc{\mmbox}[1]{\mbox{\ #1\ }} \nc{\dep}{\mrm{dep}} \nc{\fp}{\mrm{FP}}
\nc{\rchar}{\mrm{char}} \nc{\End}{\mrm{End}} \nc{\Fil}{\mrm{Fil}}
\nc{\Mor}{Mor\xspace} \nc{\gmzvs}{gMZV\xspace}
\nc{\gmzv}{gMZV\xspace} \nc{\mzv}{MZV\xspace}
\nc{\mzvs}{MZVs\xspace} \nc{\Hom}{\mrm{Hom}} \nc{\id}{\mrm{id}}
\nc{\im}{\mrm{im\,}} \nc{\incl}{\mrm{incl}} \nc{\map}{\mrm{Map}}
\nc{\mchar}{\rm char} \nc{\nz}{\rm NZ} \nc{\supp}{\mathrm Supp}

\nc{\Alg}{\mathbf{Alg}} \nc{\Bax}{\mathbf{Bax}} \nc{\bff}{\mathbf f}
\nc{\bfk}{{\bf k}} \nc{\bfone}{{\bf 1}} \nc{\bfx}{\mathbf x}
\nc{\bfy}{\mathbf y}
\nc{\base}[1]{\bfone^{\otimes ({#1}+1)}} 
\nc{\Cat}{\mathbf{Cat}}

\nc{\detail}{\marginpar{\bf More detail}
    \noindent{\bf Need more detail!}
    \svp}
\nc{\Int}{\mathbf{Int}} \nc{\Mon}{\mathbf{Mon}}
\nc{\rbtm}{{shuffle }} \nc{\rbto}{{Rota-Baxter }}
\nc{\remarks}{\noindent{\bf Remarks: }} \nc{\Rings}{\mathbf{Rings}}
\nc{\Sets}{\mathbf{Sets}} \nc{\wtot}{\widetilde{\odot}}
\nc{\wast}{\widetilde{\ast}} \nc{\bodot}{\bar{\odot}}
\nc{\bast}{\bar{\ast}} \nc{\hodot}[1]{\odot^{#1}}
\nc{\hast}[1]{\ast^{#1}} \nc{\mal}{\mathcal{O}}
\nc{\tet}{\tilde{\ast}} \nc{\teot}{\tilde{\odot}}
\nc{\oex}{\overline{x}} \nc{\oey}{\overline{y}}
\nc{\oez}{\overline{z}} \nc{\oef}{\overline{f}}
\nc{\oea}{\overline{a}} \nc{\oeb}{\overline{b}}
\nc{\weast}[1]{\widetilde{\ast}^{#1}}
\nc{\weodot}[1]{\widetilde{\odot}^{#1}} \nc{\hstar}[1]{\star^{#1}}
\nc{\lae}{\langle} \nc{\rae}{\rangle}
\nc{\lf}{\lfloor}
\nc{\rf}{\rfloor}

\nc{\QQ}{{\mathbb Q}}
\nc{\RR}{{\mathbb R}} \nc{\ZZ}{{\mathbb Z}}

\nc{\cala}{{\mathcal A}} \nc{\calb}{{\mathcal B}}
\nc{\calc}{{\mathcal C}}
\nc{\cald}{{\mathcal D}} \nc{\cale}{{\mathcal E}}
\nc{\calf}{{\mathcal F}} \nc{\calg}{{\mathcal G}}
\nc{\calh}{{\mathcal H}} \nc{\cali}{{\mathcal I}}
\nc{\call}{{\mathcal L}} \nc{\calm}{{\mathcal M}}
\nc{\caln}{{\mathcal N}} \nc{\calo}{{\mathcal O}}
\nc{\calp}{{\mathcal P}} \nc{\calr}{{\mathcal R}}
\nc{\cals}{{\mathcal S}} \nc{\calt}{{\mathcal T}}
\nc{\calu}{{\mathcal U}} \nc{\calw}{{\mathcal W}} \nc{\calk}{{\mathcal K}}
\nc{\calx}{{\mathcal X}} \nc{\CA}{\mathcal{A}}

\nc{\fraka}{{\mathfrak a}} \nc{\frakA}{{\mathfrak A}}
\nc{\frakb}{{\mathfrak b}} \nc{\frakB}{{\mathfrak B}}
\nc{\frakD}{{\mathfrak D}} \nc{\frakF}{\mathfrak{F}}
\nc{\frakf}{{\mathfrak f}} \nc{\frakg}{{\mathfrak g}}
\nc{\frakH}{{\mathfrak H}} \nc{\frakL}{{\mathfrak L}}
\nc{\frakM}{{\mathfrak M}} \nc{\bfrakM}{\overline{\frakM}}
\nc{\frakm}{{\mathfrak m}} \nc{\frakP}{{\mathfrak P}}
\nc{\frakN}{{\mathfrak N}} \nc{\frakp}{{\mathfrak p}}
\nc{\frakS}{{\mathfrak S}} \nc{\frakT}{\mathfrak{T}}
\nc{\frakX}{{\mathfrak X}} \nc{\frakx}{\mathfrak{x}}

\nc{\BS}{\mathbb{S
}}

\font\cyr=wncyr10 \font\cyrs=wncyr7
\nc{\li}[1]{\textcolor{red}{Li:#1}}
\nc{\tian}[1]{\textcolor{blue}{Tianjie: #1}}
\nc{\xing}[1]{\textcolor{purple}{Xing: #1}}

\nc{\ID}{\mathfrak{I}} \nc{\lbar}[1]{\overline{#1}}
\nc{\bre}{{\rm bre}} \nc{\sd}{\cals} \nc{\rb}{\rm RB}
\nc{\A}{\rm angularly decorated\xspace} \nc{\LL}{\rm L}
\nc{\w}{\rm wid} \nc{\arro}[1]{#1} \nc{\oneh}{\mathbf{1}_H}
\nc{\onek}{\mathbf{1}_\bfk} \nc{\onew}{\mathbf{1}}
\nc{\ver}{\rm ver} \nc{\revise}[1]{\textcolor{purple}{#1x}}

\begin{document}

\title[Connected cofiltered coalgebras and Hopf algebras]{A note on connected cofiltered coalgebras, conilpotent coalgebras and Hopf algebras}
%
\author{Xing Gao}
\address{School of Mathematics and Statistics, Key Laboratory of Applied Mathematics and Complex Systems, Lanzhou University, Lanzhou, Gansu 730000, P.\,R. China}
         \email{gaoxing@lzu.edu.cn}

\author{Li Guo}
\address{Department of Mathematics and Computer Science,
         Rutgers University,
         Newark, NJ 07102, USA}
\email{liguo@rutgers.edu}

\date{\today}
\begin{abstract}
This notes gives a proof that a connected coaugmented cofiltered coalgebra is a conilpotent coalgebra and thus a connected coaugmented cofiltered bialgebra is a Hopf algebra. This applies in particular to a connected coaugmented cograded coalgebra and a connected coaugmented cograded bialgebra.
\end{abstract}

\subjclass[2010]{
16W99, 
16T10, 
16T05.  
}

\keywords{Bialgebra, Hopf algebra, Graded, Filtered, Connected\\
\quad Appeared in: {\em Southeast Asian Bulletin of Math} {\bf 43} (2019) 313-321.}

\maketitle

\tableofcontents

\setcounter{section}{0}

\allowdisplaybreaks

\section{Introduction}
\mlabel{sec:intr}

In this note we revisit the notion of a connected filtered coalgebra, and the fact that a connected filtered bialgebras is a Hopf algebra. These notion and result and their graded variation is fundamental in the study of a large number of Hopf algebras, including the Hopf algebra of Feynman graphs in the Connes-Kreimer approach to renormalization of perturbative quantum field theory~\cite{CK}. It is one of the most used methods to obtain Hopf algebras in combinatorics and algebra
and is included in several references. There are variations in the precise meaning of the notion
of connected filtered coalgebra (see for example~\cite{Ma}) and how the connectedness implies the Hopf
property in the literature, such as~\cite{GPZ,HGK,Ma,ZGG}. As noted in~\cite{Gr}, the graded and filtered properties
of a coalgebra should include their compatibility with the counit, as well as the coproduct. Further,
what works for connected graded bialgebras becomes quite subtle for connected filtered bialgebras. Moreover, there is a closely related concept called conilpotency of a bialgebra which also
implies the Hopf algebra property~\cite[1.2.4;1.3.4]{LV}.

By elucidating these points in this note, we present a proof that a coaugmented connected filtered coalgebra is conilpotent. Hence a connected filtered bialgebra is a connected conilpotent bialgebra and hence is a Hopf algebra.

{\bf Convention} Throughout this paper, algebras are associative unitary algebras, and algebras and coalgebras are taken to be
over a commutative unitary ring $\bfk$, as are the linear maps and tensor products.

\section{From cofiltered coalgebras to conilpotent coalgebras}
\mlabel{sec:cofil}

In this section, we show that a connected coaugmented cofiltered coalgebra is a conilpotent coalgebra.

We start with some basic concepts on coalgebras. For further details, see standard references such as~\mcite{Ab,LV,Sw}. We include the notations $\beta_\ell, \beta_r$~\cite{Gub} for later applications.

\begin{defn}
A {\bf coalgebra} is a triple $(C, \Delta, \varepsilon)$ consisting of a $\bfk$-module $C$, and liner maps $\Delta: C\to C\ot C$, called
the {\bf coproduct}, and $\varepsilon:C\to \bfk$, called the {\bf counit}, that make the following diagrams commute:
\begin{equation}
\text{(Coassociativity)} \quad \xymatrix{
  C \ar[r]^{\Delta} \ar[d]_{\Delta} & C\ot C \ar[d]^{\id \ot \Delta} \\
 C\ot C \ar[r]^{\Delta\ot \id } & C\ot C\ot C             }
\mlabel{eq:coass}
\end{equation}
\begin{equation}
\text{(Counicity)} \quad
\xymatrix{
 \bfk \ot C & C\ot C \ar[l]_{\varepsilon\ot \id} \ar[r]^{\id \ot \varepsilon}& C\ot \bfk \\
  &C\ar[u]_{\Delta} \ar[lu]^{\beta_\ell} \ar[ur]_{\beta_r}&
 }
 \mlabel{eq:counit}
\end{equation}
Here $\beta_\ell$ and $\beta_r$ are linear isomorphisms given by
\begin{align*}
\beta_\ell:& C\to \bfk \ot C, \quad c \mapsto 1_\bfk \ot c,\\
\beta_r:&  C \to C \ot \bfk, \quad c \mapsto c\ot 1_\bfk, \,\, k\in \bfk, c\in C.
\end{align*}
\end{defn}

Using Sweedler's notations, we have
$$\Delta(x)=\sum_{(x)}x_{(1)}\ot x_{(2)}.$$
By iteration, we define
$${\Delta}^{k}(x): = (\id \ot {\Delta}^{k-1}){\Delta}(x), k\geq 2$$
and denote
$$ \Delta^k(x)=\sum_{(x)}x_{(1)}\ot \cdots \ot x_{(k+1)}\in C^{\ot (k+1)}.$$

\begin{defn}
A coalgebra $(C, \Delta, \varepsilon)$ is called {\bf coaugmented} if there is a linear map $u:\bfk \to C$, called the {\bf coaugmentation}, such
that $\varepsilon\, u = \id_\bfk$.
\end{defn}

\begin{lemma}
Let $(C,\Delta,\vep, u)$ be a coaugmented coalgebra.
Then
$$C = \im (u\,\vep)\oplus \ker (u\,\vep) = \im u \oplus \ker \vep.$$
\mlabel{lem:sum}
\end{lemma}

\begin{proof}
It follows from $\vep u=\id_\bfk$ that $\vep$ is surjective, $u$ is injective and $u\vep$ is idempotent. Then idempotency gives the linear decomposition
$$C=\im (u\,\vep)\oplus \ker (u\,\vep).$$
Further $\im (u\,\vep)=\im u$ since $\vep$ is surjective and $\ker (u\vep)=\ker \vep$ since $u$ is injective. Therefore,
$C=\im u\oplus \ker \vep$.
\end{proof}

\begin{lemma}
Let $(C,\Delta,\vep, u)$ be a coaugmented coalgebra.
Then for  $x\in \ker \varepsilon$,
\begin{equation}
\Delta(x) = x\ot u(1_\bfk) + u(1_\bfk) \ot x + \sum_{(x)} x'\ot x'',
\mlabel{eq:redcop}
\end{equation}
where $ x'\ot x''\in \ker\varepsilon\ot \ker\varepsilon$.
\mlabel{lem:pffil0}
\end{lemma}

\begin{proof}
By Lemma~\mref{lem:sum}, we have $C = \im u \oplus \ker\varepsilon $ and
$$C \ot C = (\im u \ot \im u) \oplus (\im u \ot  \ker\varepsilon) \oplus  (\ker\varepsilon\ot \im u) \oplus (\ker\varepsilon \ot \ker\varepsilon).$$
Now $\im u = \bfk\, u(1_\bfk)$ and so we can write
\begin{equation}
\Delta(x) = u(k)\ot u(n) + y\ot u(1_\bfk) + u(1_\bfk)\ot z + \sum_{(x)} x'\ot x'',
\mlabel{eq:swee1}
\end{equation}
where $k, n\in \bfk$, $y,z\in \ker\varepsilon$ and $x'\ot x''\in \ker\varepsilon \ot \ker \varepsilon$.

By the commutativity of the left triangle of the counicity property in Eq.~(\mref{eq:counit}), we obtain
\begin{equation*}
\begin{aligned}
 x &=\ \beta^{-1}_\ell (\vep \ot \id_C) \Delta(x) \\
 &=  \beta^{-1}_\ell ( \vep \ot \id_C) \Big(u(k)\ot u(n) + y\ot u(1_\bfk) + u(1_\bfk)\ot z +  \sum_{(x)} x'\ot x''\Big) \\
 &=\ \beta^{-1}_\ell\Big(\varepsilon(u(k)) \ot u(n) + 1_\bfk\ot z\Big)\\
  &= \beta^{-1}_\ell\Big( 1_\bfk \ot k u(n) + 1_\bfk\ot z\Big)\\
  &= ku(n) + z,
 \end{aligned}
\end{equation*}
whence $x-z = ku(n)=u(kn)$. Since $x, z\in \ker\varepsilon$, $u(kn)\in \im u$ and $\im u\cap \ker\varepsilon = 0$,
we have $x - z = 0 = u(kn)$. Hence  $x = z$ and $u(kn)=0$. Then $$u(k)\ot u(n)=u(k)\ot nu(1_\bfk)=u(k)n\ot u(1_\bfk)=u(kn)\ot u(1_\bfk)=0.$$
Similarly, by the commutativity of the right triangle in the counicity property, we get $x = y$. Thus
\begin{equation*}
\Delta(x) = x\ot u(1_\bfk) + u(1_\bfk)\ot x + \sum_{(x)} x'\ot x'',
\end{equation*}
where $x'\ot x''\in \ker\varepsilon \ot \ker \varepsilon$.
\end{proof}

Let $(C,\Delta,\vep, u)$ be a coaugmented coalgebra. By Lemma~\mref{lem:pffil0}, it makes sense to define the {\bf reduced coproduct}
$$\bar{\Delta}: \ker \varepsilon \to \ker \varepsilon \ot \ker \varepsilon, \quad x\mapsto \Delta(x) - x\ot u(1_\bfk) - u(1_\bfk)\ot x\,\, \text{ for all } x\in C.$$

{\em For the remainder of the paper}, we use the shorthand notations
\begin{equation}
\bar{\Delta}(x) = \sum_{(x)}x'\ot x'' = \sum_{(x)}x^{(1)}\ot x^{(2)}.
\mlabel{eq:nobd}
\end{equation}
In general, define $\bar{\Delta}^{k}: = (\id \ot \bar{\Delta}^{k-1})\bar{\Delta}$ for $k\geq 2$. Then
\begin{equation}
\bar{\Delta}^n (x) = \sum_{(x)} x^{(1)} \ot \cdots \ot x^{(n+1)}.
\mlabel{eq:nobd1}
\end{equation}

Now we introduce two kinds of coalgebras, namely cograded coalgebras and cofiltered coalgebras by adapting the usual notions from~\mcite{Gr,LV}.

\begin{defn}
\begin{enumerate}
\item A coalgebra $(C, \Delta, \varepsilon)$ is called {\bf cograded} if there is a grading $C = \bigoplus_{n\geq 0} C^{(n)}$ of $\bfk$-modules that is compatible with the coproduct $\Delta$ \emph{and} the counit $\vep$ in the sense that
$$\Delta(C^{(n)})\subseteq\bigoplus\limits_{p+q=n} C^{(p)}\otimes C^{(q)}\,\text{ for }\, n\geq 0, \quad \ker \vep = \oplus_{n\geq 1} C^{(n)}.$$
Elements  $x$ in $C^{(n)}$ are said to have {\bf degree} $n$, denoted by $\deg(x) = n$. \mlabel{it:pfco1}

\item A coaugmented cograded coalgebra $(C, \Delta, \varepsilon, u)$ is called {\bf connected} if $C^{(0)}=\im u$ (hence $C=\im u\oplus (\oplus_{n\geq 1} C^(n))$). \mlabel{it:pfco3x}
\end{enumerate}
\mlabel{defn:cograd}
\end{defn}

\begin{remark}
As noted in~\mcite{Gr},  it is not true that $C = C^{(0)} \oplus \ker\vep$ and $C = C^{(0)}\oplus (\oplus_{n\geq 1} C^{(n)})$
imply $\ker\vep = \oplus_{n\geq 1} C^{(n)}$, as was often taken for granted in the literature. This is simply because
$C = A \oplus B = A \oplus D$ does not imply $B=D$. So the condition $\bigoplus_{n\geq 1} C^{(n)}\subseteq \ker \vep$ is added in the definition of a cograded coalgebra.
\end{remark}

\begin{lemma}
For a connected coaugmented cograded coalgebra $(C,\Delta,\vep,u)$, we have $\ker \vep =\oplus_{n\geq 1}C^{(n)}$.
\end{lemma}
\begin{proof}
Since $\bigoplus_{n\geq 1}C^{(n)} \subseteq \ker\vep$, by the modular law, we have
$$ \ker \vep =\ker\vep \cap (\im u\oplus (\oplus_{n\geq 1}C^{(n)})) =(\ker \vep\cap \im u)\oplus (\oplus _{n\geq 1}C^{(n)}) =\oplus_{n\geq 1}C^{(n)}.$$
\end{proof}

As a more general concept, we have
\begin{defn}
\begin{enumerate}
\item A coaugmented coalgebra $(C, \Delta, \varepsilon,u)$ is called {\bf cofiltered} if there are $\bfk$-submodules $C^{n}$, $n\geq 0$, of $C$
such that
\begin{enumerate}
\item $C= \bigcup\limits^{\infty}_{n=0}C^{n}$;
\item $C^n \subseteq C^{n+1}$ for $n\geq 0$;
\item $\Delta(C^{n})\subseteq \sum \limits_{p+q=n} C^{p}\otimes C^{q}$ \,\text{ for all } $n\geq 0$;
\item $C^n =\im u \oplus (C^n\cap\ker\varepsilon)$.
\end{enumerate}
Elements  $x\in C^{n}\setminus C^{n-1}$ are said to have {\bf degree} $n$, denoted by $\deg(x)=n$. \mlabel{it:cofila}

\item A coaugmented cofiltered coalgebra $(C,\Delta,\vep,u)$ is called {\bf connected} if $C^{0}=\im u$.
\end{enumerate}
\mlabel{defn:cofil}
\end{defn}

\begin{remark}
A (connected) cograded coalgebra is a (connected) cofiltered coalgebra with the filtration defined by
$$C^{n}=  \bigoplus_{k\leq n} C^{(k)}\,\text{ for }\, n\geq 0.$$
So we will mostly focus on cofiltered coalgebras.
\mlabel{rem:cog2cof}
\end{remark}

Under the connected cofilteration condition, Lemma~\mref{lem:pffil0} can be strengthened:

\begin{lemma}
Let $(C,\Delta,\vep,u)$ be a connected coaugmented cofiltered coalgebra.
Then for  $x\in \ker \varepsilon$,
$$\Delta(x) = x\ot u(1_\bfk) + u(1_\bfk) \ot x + \sum_{(x)} x'\ot x'',$$
where $ x'\ot x''\in \ker\varepsilon\ot \ker\varepsilon$ and $0 <\deg(x') < \deg(x), 0 <\deg(x'') < \deg(x)$.
\mlabel{lem:pffil}
\end{lemma}

\begin{proof}
By Lemma~\mref{lem:pffil0}, for $x\in C$ we have
\begin{equation*}
\Delta(x) = x\ot u(1_\bfk) + u(1_\bfk)\ot x + \sum_{(x)} x'\ot x'',
\end{equation*}
where $x'\ot x''\in \ker\varepsilon \ot \ker \varepsilon$.

Let $x\in C$ with $\deg(x) = n\geq 1$. Since $x',x''\in \ker \varepsilon$, we have $x',x''\notin C^0$ by $C^0\cap \ker\varepsilon = 0$ and so $\deg(x'), \deg(x'')>0$.
By Definition~\mref{defn:cofil}~(\mref{it:cofila}), we get
$$x'\in C^p\,\text{ and } x''\in C^q\,\text{ with } p+q = n=\deg(x).$$
Using the definition of degree in Definition~\mref{defn:cofil}, we obtain
$$\deg(x')+\deg(x'')\leq p+q = \deg(x)\,\text{ and so }\, \deg(x'), \deg(x'') < \deg(x)$$
by $\deg(x'), \deg(x'')>0$. This completes the proof.
\end{proof}

We next relate these concepts to the conilpotency of  coalgebras~\cite[Section~1.2.4]{LV}. Compare with the related notion in~\cite[\S~5.2]{Mo}.

\begin{defn}
Let $(C, \Delta, \varepsilon,u)$ be a coaugmented coalgebra.
\begin{enumerate}
\item The {\bf coradical filtration} on $C$ is defined by
\begin{align*}
F_0C:= \im\,u \,\text{ and }\, F_n C:= \im\,u\oplus \{x\in \ker \varepsilon \mid \bar{\Delta}^k (x) = 0 \,\text{ for }\, k\geq n \} \, \text{ for }\, n\geq 1.
\end{align*}

\item $C$ is said to be {\bf conilpotent} if the filtration is exhaustive, that is $C = \cup_{n\geq 0} F_n C$. Equivalently, $\ker \vep =\cup_{n\geq 1} \ker \bar{\Delta}^n$.
\end{enumerate}
\mlabel{defn:conil}
\end{defn}

\begin{theorem}
Let $(C,\Delta,\vep,u)$  be a connected coaugmented cofiltered coalgebra. Then $C$ is conilpotent.
\mlabel{thm:cofcon}
\end{theorem}

\begin{proof}
Let the filtration of $C$ be given by $C = \cup_{n\geq 0} C^{n}.$
We proceed to prove $\bar{\Delta}^k(C^{n}) = 0$ for all $k\geq n\geq 1$ by induction on $n$.
When $n=1$, we have $\bar{\Delta}(x) = 0$ for all $x\in C^{1}$ by Lemma~\mref{lem:pffil0}.
When $n\geq 2$, let $x\in C^n$. Then $\deg(x)\leq n$.
By Lemma~\mref{lem:pffil},
\begin{equation}
\bar{\Delta}(x) =  \sum_{(x)} x'\ot x'',
\mlabel{eq:deltx1}
\end{equation}
where $0<\deg(x'), \deg(x'') < \deg(x) \leq n$. So $x', x''\in C^{n-1}$.
By the induction hypothesis, we have
$$\bar{\Delta}^{k-1}(x'') = 0\,\text{ for all }\, k-1\geq n-1,$$
and so by Eq.~(\mref{eq:deltx1}),
$$\bar{\Delta}^{k}(x): = (\id \ot \bar{\Delta}^{k-1})\bar{\Delta}(x) = \sum_{(x)} x' \ot \bar{\Delta}^{k-1}(x'') = 0 \,\text{ for all }\,  k\geq n.$$
This completes the proof.
\end{proof}

Since every cograded coalgebra naturally gives a cofiltration, we obtain

\begin{coro}
Let $(C,\Delta,\vep,u)$  be a connected coaugmented cograded coalgebra. Then
\begin{enumerate}
\item for $x\in \ker \varepsilon$, we have $$\Delta(x) = x\ot u(1_\bfk) + u(1_\bfk) \ot x + \sum_{(x)} x'\ot x'',$$
where $ x'\ot x''\in \ker\varepsilon\ot \ker\varepsilon$ and $0 <\deg(x'), \deg(x'') < \deg(x)$.

\item $C$ is conilpotent.
\end{enumerate}
\mlabel{coro:cogcon}
\end{coro}

\begin{proof}
It follows from Lemma~\mref{lem:pffil}, Theorem~\mref{thm:cofcon} and Remark~\mref{rem:cog2cof}.
\end{proof}

\section{From connected bialgebras to Hopf algebras}
In this section, we prove that connected cograded bialgebras and connected cofiltered bialgebras are Hopf algebras.

\begin{defn}
A {\bf bialgebra} is a quintuple $(H, m,u, \Delta, \varepsilon)$ where $(H, m, u)$ is an algebra
and $(H, \Delta, \varepsilon)$ is a coalgebra such that $m: H\ot H\to H$ and $u:\bfk \to H$
are morphisms of coalgebras.
\end{defn}

\begin{remark}
The conditions that $m$ and $u$ are morphisms of coalgebras can be equivalent replaced by the
conditions that $\Delta:H \to H\ot H$ and $\varepsilon: H\to \bfk$ are morphisms of algebras.
\end{remark}

Since $H$ is a bialgebra, we have $\vep u=\id_{\bfk}$. Hence the unit $u:\bfk \to H$ gives a coaugmentation of the coalgebra $(H,\Delta,\vep)$. Thus Lemma~\mref{lem:sum} applies.

For an algebra $A$ and a coalgebra $C$, the {\bf convolution} of two linear maps $f,g$ in $\Hom(C, A)$ is defined to be the
map $f\ast g$ given by the composition
$$ C \overset{\Delta}{\to}  C\ot C \overset{f\ot g}{\to} A\ot A \overset{m}{\to} A.$$
In other words,
$$(f\ast g)(a) = \sum_{(a)} f(a_{(1)}) g(a_{(2)})\,\text{ for } a\in A.$$

\begin{defn}
Let $(H,m,u,\Delta, \varepsilon)$ be a bialgebra. A linear endomorphism $S$ of $H$ is called an {\bf antipode}
for $H$ if it is the inverse of $\id_H$ under the convolution product: $$S\ast \id_H = \id_H\ast S = u\varepsilon.$$
A {\bf Hopf algebra} is a bialgebra $H$ with an antipode $S$.
\end{defn}

The following is the main theorem of this paper. Note that the grading and filtration are required to be compatible with the coproduct $\Delta$ and the counit $\vep$, not with the product $m$ or the unit $u$. So the result can be applied to a bialgbera with a grading or filtration which is not necessarily compatible with the product.

\begin{theorem}
Let $(H,m,u,\Delta,\vep)$  be a bialgebra such that $(H,\Delta,\vep,u)$ is a connected coaugmented cofiltered (in particular cograded) coalgebra.
Then $H$ is a Hopf algebra and the antipode $S$ is given by
\begin{equation}
S(\oneh) = \oneh\,\text{ and }\, S(x) = -x + \sum_{n\geq 1} (-1)^{n+1} m^n \bar{\Delta}^{n}(x)\,\text{ for }\, x\in \ker\varepsilon.
\mlabel{eq:hypos1}
\end{equation}
\mlabel{thm:filhop}
\end{theorem}

\begin{proof}
Observe from~\cite[p.~20]{LV} that a conilpotent bialgebra automatically has an antipode $S$ given by Eq.~(\mref{eq:hypos1}).
Then the result follows from Theorem~\mref{thm:cofcon} and Corollary~\mref{coro:cogcon}.
\end{proof}

Recall our aforementioned notations for $\bar{\Delta}^n (x)$ in Eqs.~(\mref{eq:nobd}) and~(\mref{eq:nobd1}).

\begin{remark}
The formula
for the antipode in Eq.~(\mref{eq:hypos1}) can be restated as
\begin{equation*}
S(\oneh) = \oneh\,\text{ and }\, S(x) = -x + \sum_{n\geq 1} \sum_{(x)} (-1)^{n+1} x^{(1)} \cdots x^{(n+1)} \,\text{ for }\, x\in \ker\varepsilon,
\end{equation*}
which coincides with the following recursive formula on degree:
\begin{equation*}
S(\oneh) = \oneh\,\text{ and }\, S(x) = -x - \sum_{(x)} S(x') x'' = -x - \sum_{(x)} x' S(x'')  \,\text{ for }\, x\in \ker\varepsilon.
\end{equation*}
Here $ x'\ot x''\in \ker\varepsilon\ot \ker\varepsilon$ and $0<\deg(x'), \deg(x'') < \deg(x)$ by Corollary~\mref{coro:cogcon} (resp. Lemma~\mref{lem:pffil}).
This coincidence can be seen by induction on degree.
Obviously, they agree on the initial step,
that is $S(\oneh) = \oneh$. For the induction step,
\begin{align*}
-x - \sum_{(x)} S(x') x'' =& -x - \sum_{(x)}   \Big( -x' + \sum_{n\geq 1}\sum_{(x')} (-1)^{n+1}  x'^{(1)} \cdots x'^{(n+1)} \Big) x''\\
=&  -x + \sum_{(x)} \Big( (-1)^2x'x''   +\sum_{n\geq 1} \sum_{(x')} (-1)^{n+2}  x'^{(1)} \cdots x'^{(n+1)} x'' \Big)\\
=&  -x + \sum_{(x)} \Big( (-1)^2x'x''   +\sum_{n\geq 1} (-1)^{n+2}  x^{(1)} \cdots x^{(n+2)}\Big)\\
=&  -x + \sum_{(x)} \Big( (-1)^2  x^{(1)}  x^{(2)}   +\sum_{n\geq 2} (-1)^{n+1}  x^{(1)} \cdots x^{(n+1)}\Big)\\
=&  -x + \sum_{(x)} \sum_{n\geq 1} (-1)^{n+1}  x^{(1)} \cdots x^{(n+1)}\\
=&  -x + \sum_{n\geq 1}\sum_{(x)} (-1)^{n+1}  x^{(1)} \cdots x^{(n+1)},
\end{align*}
as required. Similarly, we can show the case of $ S(x) =-x - \sum_{(x)} x' S(x'')$.
\end{remark}

\smallskip

\noindent {\bf Acknowledgements}: This work was supported by the National Natural Science Foundation of China (Grant No. 11771190), Fundamental Research Funds for the Central Universities (No.~lzujbky-2017-162), the NSF of Gansu Province (No.~17JR5RA175) and NSF of Shandong Province (No. ZR2016AM02).

\end{document}